\documentclass[12pt, a4paper,leqno]{amsart}

\usepackage[utf8]{inputenc}
\usepackage[T1]{fontenc}
\usepackage{amsmath}
\usepackage{amsthm}
\usepackage{amsfonts}
\usepackage{amssymb}
\usepackage[nobysame]{amsrefs}
\usepackage{url}
\usepackage{esint}
\usepackage{verbatim}
\usepackage{graphicx,epic,eepic}
\usepackage[dvipsnames]{xcolor}
\usepackage[colorlinks, allcolors=RedViolet,pdfstartview=,pdfpagemode=UseNone]{hyperref} %backref
\usepackage[many]{tcolorbox}
\usepackage{multicol}

\theoremstyle{plain}
\newtheorem{thm}[equation]{Theorem}
\newtheorem{lem}[equation]{Lemma}

\newtheorem{cor}[equation]{Corollary}

\theoremstyle{definition}
\newtheorem{defn}[equation]{Definition}

\theoremstyle{remark}
\newtheorem{rem}[equation]{Remark}

\numberwithin{equation}{section}

\definecolor{boxcolor}{RGB}{227,255,227}
\definecolor{boxframecolor}{RGB}{227,255,227}
\newtcolorbox[auto counter]{pbox}[2][]{%
top=-2pt,bottom=2pt,left=2pt,right=2pt,breakable,enhanced,
colback=white,colframe=ForestGreen,
floatplacement=t,float,
title=\textsc{Box}~\thetcbcounter. #2, #1}

\DeclareMathOperator{\esssup}{ess\,sup}

\DeclareMathOperator{\essinf}{ess\,inf}

\newcommand{\R}{\mathbb{R}}

\newcommand{\Rn}{\mathbb{R}^n}

\newcommand{\B}{\mathcal{B}}

\newcommand{\ainc}[1]{\hyperref[defn:aInc]{{\normalfont(aInc){\ensuremath{_{#1}}}}}}
\newcommand{\inc}[1]{\hyperref[defn:aInc]{{\normalfont(Inc){\ensuremath{_{#1}}}}}}
\newcommand{\adec}[1]{\hyperref[defn:aDec]{{\normalfont(aDec){\ensuremath{_{#1}}}}}}
\newcommand{\azero}{\hyperref[defn:A0]{\textcolor{RedViolet}{\normalfont(A0)}}}
\newcommand{\aone}{\hyperref[defn:A1]{\textcolor{RedViolet}{\normalfont(A1)}}}
\newcommand{\aonew}{\hyperref[defn:A1w]{\textcolor{RedViolet}{\normalfont(A1)$_\omega$}}}
\newcommand{\aones}[1]{\hyperref[defn:A1psi]{\textcolor{RedViolet}{\normalfont(A1-\ensuremath{#1})}}}
\newcommand{\atwo}{\hyperref[defn:A2]{\textcolor{RedViolet}{\normalfont(A2)}}}
\newcommand{\atwow}{\hyperref[defn:A2w]{\textcolor{RedViolet}{\normalfont(A2)$_\omega$}}}

\def\essinf{\operatornamewithlimits{ess\,inf}}

\def\esssup{\operatornamewithlimits{ess\,sup}}

\def\loc{{\rm loc}}

\def\R{{\mathbb{R}}}
\def\Rn{{\mathbb{R}^n}}

\newcommand\norm[1]{\left\lVert#1\right\rVert}

\def\phi{\varphi}

\def\phiw{\Phi_\textnormal{w}}
\def\phic{\Phi_\textnormal{c}}

\newcommand\blfootnote[1]{
    \begingroup
    \renewcommand\thefootnote{}\footnote{#1}
    \addtocounter{footnote}{-1}
    \endgroup
}

\author{Vertti Hietanen}

\title[Sufficient condition for boundedness]{Sufficient condition for boundedness of maximal operator on weighted generalized Orlicz spaces}

\address{Department of Mathematics and Statistics, FI-00014 University of Helsinki, Finland}

\email{vertti.hietanen@helsinki.fi}

\begin{document}

\begin{abstract}
    We prove that the Hardy-Littlewood maximal operator is bounded in the weighted generalized Orlicz space if the weight satisfies the classical Muckenhoupt condition $A_p$ and $t \to \frac{\phi(x,t)}{t^p}$ is almost increasing in addition to the standard conditions. 
    \end{abstract}

\date{\today}



\maketitle 

\blfootnote{The author is supported by the Finnish Cultural Foundation.}

\section{Introduction}

Generalized Orlicz spaces have been studied since the 1940s. They are also known as Musielak--Orlicz spaces due to more comprehensive presentation of J. Musielak in his monograph published in 1983 \cite{musielak}. Intuitively generalized Orlicz space $L^\phi$ consists of all measurable functions $f$ such that
\[
 \varrho_\phi(f) :=\int_\Rn \phi(x, |f(x)|)\, dx<\infty.
\]
We call the function $\varrho_\phi(f)$ a \textit{modular}. If the function $\phi$ does not depend directly on $x$, $\phi(x,t) = \phi(t)$, 
then we obtain \textit{Orlicz spaces}. Function $\phi(x,t)= t^p$ gives us the \textit{Lebesgue space} $L^p$ and $\phi(x,t) = t^{p(x)}$, the \textit{variable exponent Lebesgue space} $L^{p(\cdot)}$. Other important space is given by the double phase functional  $\phi(x,t) = t^{p} + a(x) t^q$, where $q>p$. Apart from being a natural generalization for many well-researched function spaces, the study of generalized Orlicz spaces has applications to image processing, fluid dynamics, and differential equations; see, e.g., \cite{fluid, BD, PDE, bookPolish, HastoOk2022}.

The Hardy--Littlewood maximal operator $M$ is a central tool in harmonic analysis. 
It was proved by D. Gallardo in 1988 \cite{Gal88} that $M:L^\phi\to L^\phi$ for an Orlicz function 
$\phi(x,t)=\phi(t)$ if and only if $\phi$ satisfies \ainc{p} with $p>1$. For an Orlicz function, condition \ainc{p} implies the existence of an equivalent convex function $\psi^{1/p}$ (see Lemma 2.2.1, \cite{Og}), for which a Jensen-type inequality (Lemma \ref{4.3.1}) can be obtained. Jensen's inequality is a helpful tool in obtaining boundedness results, but the inequality does not hold as such in the generalized Orlicz space. 

In 2004, L. Diening proved the boundedness of $M$ locally in the variable exponent, with the exponent $p$ being log-Hölder continuous \cite{Dtrick}. The boundedness was a result of a trick that $\phi(x,t)=t^{p(x)}$ satisfies Jensen's inequality up to an error term independent of $t$. The technique was soon generalized to the global case by D. Cruz-Uribe, A. Fiorenza, and C.J. Neugebauer \cite{Cruz2003}, and independently by A. Nekvinda \cite{Nekvinda}. In the book \cite{ves}, Diening's trick was named as \textit{Key estimate} and formulated into Theorem 4.2.4. 

In 2013, F-Y. Maeda, Y. Mizuta, T. Ohno and T. Shimomura
\cite{MaeMOS13a} provided the first sufficient conditions for boundedness of $M$ in
generalized Orlicz space using rather heavy machinery.
In 2015, P. Hästö \cite{H15} proved a sharper version with $\phi$ satisfying conditions \azero{}-\atwo{} and \ainc{} and simplified the proof. The proof was based on generalized version of the key estimate that was obtained by using \aone{} to estimate the generalized $\phi$ function with a regular Orlicz function locally and then using Jensen-type inequality \ref{4.3.1}. Finally, we note that the condition \azero{} excludes weighted norm
inequalities, both in the classical case $\phi(x,t)=t^p \omega(x)$ and in
the variable exponent case $\phi(x,t)=t^{p(x)}\omega(x)$.  In either case,
the condition \azero{} requires the weight to be essentially constant.

While generalized Orlicz spaces have been actively studied during the last twenty years, weighted generalized Orlicz spaces have  got only a little attention. The weighted space is defined with a weighted modular
$$\int_\Rn \varphi (x, | f (x) |) \omega (x) \, dx,$$ where $\varphi$ satisfies \azero{} and $\omega$ is a weight function. The crucial question is to give a property for the weight $\omega$ so that the Hardy-Littlewood maximal function is bounded. In 1972, B. Muckenhoupt \cite{muck} defined a class of weight functions, denoted by $A_p$, for which the Hardy-Littlewood maximal function is bounded in the weighted Lebesgue space $L^p(\Rn, \omega)$. The functions in this class bear the name \textit{Muckenhoupt weights}. In variable exponent spaces, power-type weights have been studied by V. Kokilashvili N. Samko and S. Samko; see, e.g. \cite{s1, kokilashvili2007general, kokilashvili2007maximal, kokilashvili2007singular}. Muckenhoupt-type weights have been studied by D. Cruz-Uribe, L. Diening, and P. Hästö \cite{CDH11}. The latter two authors have presented an alternative approach in preprint \cite{DP08}, where the weight $\omega$ is treated as a measure instead of a multiplier. 

It was proved by A. Gogatishvili and V. Kokilashvili \cite{crit} in 1994  that the maximal operator $M$ is bounded in weighted Orlicz space if and only if $\phi$ satisfies \ainc{p} and $\omega\in A_p$. This condition for weight $\omega$ is also known to be sufficient in the variable exponent space \cite{DP08}.
In this paper we prove the sufficiency of this condition in generalized Orlicz space. We use a similar approach as in \cite{DP08} by treating the weight $\omega$ as a measure. We obtain a weighted version of the \textit{key estimate} (Theorem 4.3.2 from \cite{Og}), which plays a major role in the proof of  unweighted case. This requires generalizing conditions \aone{} and \atwo{} to \aonew{} and \atwow{}, which are in line with the weighted space. At the end of this paper, we consider these new conditions with functions $\phi(x,t) = t^{p(x)}$ and $\phi(x,t) := t^{p} + a(x) t^q$.

\section{Preliminaries}

Let $\Omega \subset \Rn$ be an open set equipped with the $n$-dimensional Lebesgue measure. By $B$ we always denote an open ball in $\Rn$. If there exists a constant $C$ such that $f(x)\leq C g(x)$ for almost every $x$, then we write $f\lesssim g$. If $g \lesssim f \lesssim g$, then we write $f \approx g$. We say that functions $\phi,\psi:\Omega \times [0, \infty)\to [0,\infty]$ are \textit{equivalent} if there exists $L\geq 1$ such that $\psi(x,\frac{1}{L}t)\leq \phi(x,t) \leq \psi(x,Lt)$. A function $f:(0, \infty) \to \R$ is called \textit{almost increasing} if there exists a constant $a\geq 1$ such that $f(s)\leq a f(t)$ for all $0<s<t$.

For a locally integrable function $f\in L^1_{\loc}(\Omega)$, we define the (Hardy-Littlewood) \textit{maximal operator} $M$ by 
\begin{equation}
\nonumber
Mf(x)=\sup_{B \ni x }\frac{1}{|B|}\int_{B\cap\Omega}|f(y)|\, dy,
\end{equation}
  where the supremum is taken over all open balls $B$ containing $x$.

\begin{defn} \label{defn:aInc}
     Let $p>0$. We say that a function $\phi:\Omega \times [0, \infty)\to [0,\infty]$ satisfies \textup{(aInc)}$_p$ if $t \to \frac{\phi(x,t)}{t^p}$ is almost increasing, i.e there exists $a\geq 1$ such that
     \begin{equation*}
         \frac{\phi(x,s)}{s^p}\leq a \frac{\phi(x,t)}{t^p}
     \end{equation*}
     for all $0<s<t$.
\end{defn}

\begin{defn}
A function $\phi:\Omega \times [0,\infty) \to [0,\infty]$ is said to be a (generalized) $\Phi$-\textit{prefunction} if $x \to \varphi(x,|f(x)|)$ is measurable for every $f \in L^0(\Omega)$ and $t \to \phi(x,t)$ is increasing with  $\varphi(x,0)=\lim_{t\to 0^+}\varphi(x,t)=0$ and $\lim_{t\to \infty}\varphi(x,t)=\infty$ for almost every $x \in \Omega$. 

We say that the $\Phi$-\textit{prefunction} is a \textit{weak} $\Phi$-\textit{function}, and write $\phi \in \phiw (\Omega)$, if it satisfies \ainc{1} on $(0,\infty)$ for a.e $x \in \Omega$. If, in addition, $\varphi \in \phiw(\Omega)$ is convex and left-continuous with respect to $t$ for almost every $x$, then we write $\phi \in \phic(\Omega)$.

\end{defn}

For $\varphi \in \phiw(\Omega)$ and $B\subset \Rn$, we define
\begin{equation*}
    \varphi_B^+(t):=\esssup_{x \in B\cap\Omega} \varphi(x,t) \quad  \text{and} \quad \varphi_B^-(t):=\essinf_{x \in B\cap\Omega} \varphi(x,t).
\end{equation*} 
We say that $\varphi \in \Phi(\Omega)$ is \textit{degenerate} if $\varphi|_{(0,\infty)}\equiv 0$ or $\varphi|_{(0,\infty)}\equiv \infty$. If $\varphi^+_B$ or $\phi^-_B$ is non-degenerate, then it is a weak $\Phi$-function (\cite{Og}, Lemma 2.5.16).

By $\omega$ we always denote a \textit{weight}, that is a non-negative and locally integrable function. We often denote $\omega(\Omega):=\int_\Omega \omega(x)dx$, and in this sense we treat $\omega$ as a measure.

\begin{defn}

Let $\varphi \in \phiw(\Omega)$ and $p \geq 0$. We say that $\varphi$ satisfies 
\begin{itemize}
    \item[(A0):] if there exists $\beta_0\in(0,1]$ such that $\phi(x, \beta_0) \leq 1\leq\phi(x,\frac{1}{\beta_0})$ for almost every $x\in \Omega$. \label{defn:A0}
    \item[(A1)$_\omega$:] \label{defn:A1w} if there exists $\beta_1\in(0,1]$ such that $\varphi(x,\beta_1 t)\leq \varphi(y,t)$ for every $\phi(y,t)\in[1,\frac{1}{\omega(B)}]$, almost every $x,y\in B\cap\Omega$ and every ball $B$ with $\omega(B)\leq1$. 
    \item[(A2)$_\omega$:] if for every $s>0$ there exist $\beta_2\in(0,1]$ and $h\in L^1(\Omega,\omega)\cap L^{\infty}(\Omega)$, $h\geq0$, such that $\varphi(x,\beta_2 t)\leq \varphi(y,t)+h(x)+h(y)$ for almost every $x,y \in \Omega$ when $\varphi(y,t)\in[0,s]$. \label{defn:A2w}
\end{itemize}
\end{defn}

If $\omega(x)=1$, then condition \aonew{}  is equivalent to the standard condition (A1$'$)  presented in the book \cite{Og}. According to book's Lemma 4.2.5, \atwow{} is equivalent to condition (A2), but the book formulation of (A2) contains a flaw that was recently corrected in \cite{rev}.

The following formulation of \atwow{} will be useful when considering the condition in variable exponent spaces.

\begin{lem}\label{a2eq}
    Let $\varphi \in \phiw(\Omega)$. Then $\varphi$ satisfies \atwow{} if and only if there exist $\varphi_{\infty}\in \phiw$, $h\in L^1(\Omega,\omega)\cap L^{\infty}(\Omega)$ and $\beta \in (0,1]$ such that
    \begin{align*}
        &\varphi(x,\beta t) \leq \varphi_\infty(t) + h(x) \quad \text{when} \quad \varphi_\infty(t)\in[0,1],  \\ 
        \text{and} 
        \quad &\varphi_\infty(\beta t) \leq \varphi(x,t) + h(x) \quad \text{when} \quad \varphi(x,t)\in [0,1]
    \end{align*}
    for almost every $x\in \Omega$.
\end{lem}
\begin{proof}
    The proof does not differ from the proofs of Lemmas 4.2.7 and 4.2.10 from \cite{Og}, which cover the case $\omega(x)=1$.
    \end{proof}

The next Jensen-type inequality is an important tool in obtaining the \textit{key estimate}, Theorem \ref{key}. 

\begin{lem}[\cite{Og}, Lemma 4.3.1]\label{4.3.1}
    Let $\varphi:[0,\infty) \to [0,\infty]$ be a prefunction that satisfies \textup{(aInc)$_p$}, $p>0$. Then there exists $\beta \in (0,1]$ such that the following inequality holds for every measurable set $U$, $|U|\in (0, \infty)$, and every $f\in L^1(U)$:
\begin{equation*}
    \varphi\left(\beta \fint_U |f|\, dx\right)^{\frac{1}{p}}\leq \fint_U \varphi(f)^{\frac{1}{p}}\, dx.
\end{equation*}
\end{lem}

Next we give an exact definition of \textit{weighted generalized  Orlicz space}; starting with the definition of modular $\varrho_\phi^\omega$.

\begin{defn}
    Let $\phi \in \phiw(\Omega)$. We define the \textit{ weighted modular} $\varrho_\phi^\omega(\cdot)$ for $f \in L^0(\Omega)$ by
    \begin{equation*}
    \varrho_\varphi^\omega(f):=\int_\Omega\varphi(x,|f(x)|)\omega(x)\, dx.
    \end{equation*}
    The \textit{weighted generalized Orlicz space}  is defined as the set 
    \begin{equation*}
        L^\varphi(\Omega, \omega):=\{f \in L^0(\Omega):\norm{f}_{L^\phi(\Omega, \omega)}<\infty\},
    \end{equation*}
    where
    \begin{equation*}
        \norm{f}_{L^\phi(\Omega, \omega)}:=\inf\bigg\{\lambda>0:\varrho^\omega_\varphi\bigg(\frac{f}{\lambda}\bigg)\leq1\bigg\}.
    \end{equation*}
\end{defn}

An useful relation between the modular and the norm is the so-called \textit{unit ball property:}

\begin{lem}[\cite{Og}, Lemma 3.2.3]\label{ubp}
    Let $\phi \in \phiw(\Omega).$ Then
\begin{equation*}
    \norm{f}_{L^\phi(\Omega, \omega)} <1 \ \implies \ \varrho^\omega_\varphi(f)\leq1 \ \implies \ \norm{f}_{L^\phi(\Omega, \omega)}\leq1.
\end{equation*}
If $\varphi$ is left-continuous, then $ \norm{f}^\omega_{L^\phi(\Omega, \omega)} \leq1 \ \iff \ \varrho^\omega_\varphi(f)\leq1 $.
\end{lem}
The unit ball property holds for any weight $\omega$ and more general measures as well. 

For more information about generalized Orlicz spaces see the monograph \cite{Og}.

\section{Properties of $A_p$ -weights}

The classical Muckenhoupt condition can be stated as follows.
\begin{defn}
Let $1< q <\infty$ and $q'$ satisfy $\frac{1}{q}+\frac{1}{q'}=1$. If a non-negative and locally integrable weight function $\omega$ satisfies
\begin{equation*}
    [\omega]_{A_q}:=\sup_{B\in \B}|B|^{-q}\omega(B)\left(\int_B \omega(x)^{-q'/q} \, dx\right)^{q-1}<\infty,
\end{equation*}
where $\B$ denotes the family of all open balls $B \subset \Rn$, then we say that $\omega \in A_q$.
\end{defn}

We define the class $A_\infty$ as the union of all classes $A_q$, $q \in [1,\infty)$, and the class $A_1$ to consist of all weights $\omega$ satisfying $M\omega \lesssim \omega$.

In a paper \cite{muck}, published in 1972, B. Muckenhoupt showed the following famous result for weights satisfying the above condition.

\begin{lem}[\cite{muck}, Theorem 9]\label{clas-max}
    Let $1<q<\infty$. Then $\omega \in A_q$ if and only if 
    \[M:L^q(\Omega,\omega)\to L^q(\Omega,\omega).\]
\end{lem}

The next lemma gives us an alternative definition for the class $A_q$, which is more closely related to boundedness of $M$. 

\begin{lem}[\cite{wn}, Theorem 1.12]\label{lem1}
Let $1< q <\infty$ and $[\omega]_{A_q}$ be the constant from $A_q$ condition. Then $\omega \in A_q$ if and only if 
\begin{equation*}\label{alt}
    \left(\frac{1}{|B|}\int_B f(x)\, dx\right)^{q}\leq \frac{[\omega]_{A_q}}{\omega(B)} \int_Bf(x)^q\omega(x) \, dx
\end{equation*}
for all measurable $f\geq 0$ and for all open balls $B\subset \Rn$.

\end{lem}

The following property of weights in $A_q$ classes is often useful in controlling $\omega(B)$ with various operations; see, e.g., Lemma \ref{wh}.

\begin{lem}\label{bxy}
    Let $1<q<\infty$. If $\omega \in A_q,$ then
    $$\omega(B(x,r))\gtrsim\omega(B(y,R))\left(\frac{r}{|x-y|+r+R}\right)^{qn},$$
    for every $x,y \in \Rn$ and $R,r>0$.
\end{lem}
\begin{proof}
    Let $x,y \in \Rn$ and $R,r>0$. Since $\omega \in A_q$, we have by Lemma \ref{clas-max} that $M:L^q(\Rn,\omega) \to L^q(\Rn,\omega)$.  Then
    \begin{align*}        \omega(B(x,r))&=\int_\Rn(\chi_{B(x,r)}(z))^{q}\omega(z)dz\gtrsim \int_\Rn(M\chi_{B(x,r)}(z))^{q}\omega(z)dz \\ 
    &=\int_\Rn \left(\sup_{B\ni z}\frac{|B\cap B(x,r)|}{|B|}\right)^{q}\omega(z)dz. 
    \end{align*}
    Let us choose $B=B(y,|x-y|+r+R)$. Then
    \begin{align*}
        \int_\Rn \left(\sup_{B\ni z}\frac{|B\cap B(x,r)|}{|B|}\right)^{q}\omega(z)dz&\geq \int_{B(y,R)} \left(\frac{|B(x,r)|}{|B(y,|x-y|+r+R)|}\right)^{q}\omega(z)dz\\
        &=\omega(B(y,R))\left(\frac{r}{|x-y|+r+R}\right)^{qn}.
    \end{align*}
\end{proof}

\begin{rem}
    We can observe that the weight $\omega \in A_q$ is doubling as a measure since by Lemma \ref{bxy} it follows that $\omega(B(x,r))\gtrsim \omega(B(x,2r))3^{-qn}.$

\end{rem}

\begin{rem}
    We can also notice that by Lemma \ref{bxy} the weight $\omega \in A_\infty$ has as at most polynomial growth i.e  $\omega(0,r)\lesssim \omega(B(0,1))(1+r)^{qn}\lesssim r^{qn}$ for some $q\in[1,\infty)$ when $r>1$.
\end{rem}

\section{Boundedness}

  Next we prove the key estimate in weighted generalized Orlicz space. We extend the techniques from \cite{Og} to the weighted case.

\begin{thm}\label{key}
    Let $\varphi \in \Phi_w(\Omega)$ satisfy \azero{}, \aonew{}, \atwow{} and \ainc{p}, with $p \in (0, \infty)$. Let $\omega \in A_p$. Then there exists $\beta>0$ and $h\in L^1(\Omega, \omega)\cap L^{\infty}(\Omega)$ such that
\begin{equation*}
    \varphi\left(x,\frac{\beta}{|B|}\int_{B\cap \Omega} |f| \, dy\right)^{\frac{1}{p}}\leq \frac{1}{|B|}\int_{B\cap\Omega} \varphi(y,|f|)^{\frac{1}{p}}\, dy+h(x)^{\frac{1}{p}}+\frac{1}{|B|}\int_{B\cap\Omega} h(y)^{\frac{1}{p}}\, dy
\end{equation*}
for every ball $B$, $x\in B\cap\Omega$ and $f\in L^{\varphi}(\Omega,\omega)$ with $\varrho_\varphi^\omega(f)\leq 1$.
\end{thm}

\begin{proof}
    Let $\beta_0$ be the constant from \azero{}, $\beta_1$ the constant from \aonew{}, and $\beta_2$ from \atwow{}. Let $\beta_J$ be the constant from the  Jensen-type inequality Lemma \ref{4.3.1}.
We may assume without loss of generality that $f\geq0$. Fix a ball $B$, denote $\hat{B}:=B \cap \Omega$ and choose $x\in \hat{B}$. Denote $f_1:=f\chi_{\{f>1/\beta_0\}}$, $f_2 := f-f_1$, and $D_i:= \frac{1}{|B|}\int_{\hat{B}} f_i \, dy$.

We first estimate $D_1$. Let $c_\omega:=\max\{1,[\omega]_{A_p}\}$. Now $\varphi$ satisfies \ainc{1} and thus $\varphi_B^-(\frac{1}{a}c_\omega^{-1}f_1)\leq c_\omega^{-1}\varphi_B^-(f_1)$. Then by Lemmas \ref{4.3.1} and \ref{lem1}, assuming $\varrho_\varphi^\omega(f_1)\leq\varrho_\varphi^\omega(f)\leq 1$, we obtain that
\begin{align*}
    \varphi_B^-(\tfrac{1}{2a}c_\omega^{-1}\beta_J D_1) &\leq \left( \frac{1}{|B|}\int_{\hat{B}} \varphi^-_B(\tfrac{1}{2a}c_\omega^{-1}f_1)^{\frac{1}{p}} \, dy \right)^p\\
    &\leq  \frac{[\omega]_{A_p}}{\omega(B)}\int_{\hat{B}} \varphi^-_B(\tfrac{1}{2a}c_\omega^{-1}f_1) \omega(y)\, dy \\
    &\leq \frac{1}{\omega(B)}\int_{\hat{B}} \frac{1}{2}\varphi^-_B(f_1) \omega(y)\, dy
    \\
    &\leq \frac{1}{2\omega(B)}\int_{\hat{B}} \varphi(y,f_1) \omega(y)\, dy < \frac{1}{\omega(B)},
\end{align*}
where we used \ainc{1} for the third inequality.
Denote $\beta'_J:=\frac{1}{2a}c_\omega^{-1}\beta_J$, and we have $\varphi_B^-(\beta'_J D_1)<\frac{1}{\omega(B)}$. We may also note that $\beta'_J\leq1$.

Suppose first that $\varphi^-_B(\beta'_JD_1)\geq 1$. This implies $\omega(B)<1$. Then there exists $y \in \hat{B}$ with $\varphi(y, \beta'_J D_1) \in [1, \frac{1}{\omega(B)}]$ and $\varphi(y,\beta'_J D_1)\leq 2 \varphi_B^-(\beta'_J D_1)$. Then by \aonew{} and Lemma \ref{4.3.1} for the third inequality we obtain
\begin{align*}
    \varphi(x,\beta_1 \beta'_J D_1)^{\frac{1}{p}}\leq \varphi(y,\beta'_J D_1)^{\frac{1}{p}}
    &\leq 2 \varphi_B^-(\beta'_J D_1)^{\frac{1}{p}}\\ &\leq \frac{2}{|B|} \int_{\hat{B}}\varphi_B^-(f_1)^{\frac{1}{p}}\, dy\leq \frac{2}{|B|}\int_{\hat{B}}\varphi(y,f_1)^{\frac{1}{p}}\, dy.
\end{align*}
Next we consider the case $\varphi^-_B(\beta'_J D_1) < 1$. By \azero{}, this implies that $\beta'_J D_1\leq \frac{1}{\beta_0}$. By \ainc{p} and \azero{}, we conclude that
\begin{equation*}
    \varphi(x, \beta_0^2\beta'_J D_1)^{\frac{1}{p}}\leq a^{\frac{1}{p}}\beta_0\beta'_J D_1 \varphi(x,\beta_0)^{\frac{1}{p}}\leq a^{\frac{1}{p}}\beta_0 \beta'_J D_1 \leq a^{\frac{1}{p}}D_1.
\end{equation*}
By \azero{}, $1 \leq \varphi(y,1/\beta_0)$. Since $f_1>1/\beta_0$, when it is non-zero, we find by \ainc{p} for the second inequality, that
\begin{equation*}
    D_1=\frac{1}{|B|}\int_{\hat{B}}f_1 \, dy \leq \frac{1}{|B|}\int_{\hat{B}} f_1 \varphi(y, \frac{1}{\beta_0})^{\frac{1}{p}}\, dy\leq \frac{a^{\frac{1}{p}}}{\beta_0|B|}\int_{\hat{B}} \varphi(y, f_1)^{\frac{1}{p}}\, dy.
\end{equation*}
In view of this and the conclusion of previous paragraph, we obtain final estimate for $f_1$
\begin{equation*}
    \varphi(x, \frac{\beta_0}{2a^{\frac{3}{p}}}\beta_0^2\beta_1\beta'_J D_1)^{\frac{1}{p}}\leq \frac{\beta_0}{2 a^{\frac{2}{p}}} \varphi(x, \beta_0^2 \beta_1 \beta'_J D_1)^{\frac{1}{p}}\leq \frac{1}{|B|}\int_{\hat{B}}\varphi(y,f_1)^{\frac{1}{p}} \, dy,
\end{equation*}
where we also used \ainc{p} for the first inequality.

Let us move to $f_2$. Since $f_2 \leq \frac{1}{\beta_0}$, by \azero{} we obtain $\varphi(y, \beta_0^2 f_2)\leq \varphi(y, \beta_0)\leq 1$. This allows us to use \atwow{}, and with Lemma \ref{4.3.1} for $\varphi(x, \cdot)^{\frac{1}{p}}$ (with $x$ fixed) we obtain 
\begin{align*}
    \varphi(x, \beta_J \beta_0^2 \beta_2 D_2)^{\frac{1}{p}}&\leq \frac{1}{|B|}\int_{\hat{B}} \varphi(x, \beta_0^2 \beta_2 f_2(y))^{\frac{1}{p}}\, dy \\
    & \leq \frac{1}{|B|}\int_{\hat{B}} \varphi(y,\beta_0^2 f_2)^{\frac{1}{p}}+h(x)^{\frac{1}{p}}+h(y)^{\frac{1}{p}}\, dy \\
    & \leq \frac{1}{|B|}\int_{\hat{B}} \varphi(y, f_2)^{\frac{1}{p}}+h(x)^{\frac{1}{p}}+h(y)^{\frac{1}{p}}\, dy.
\end{align*}
Since $\varphi^{\frac{1}{p}}$ is increasing,
\begin{align*}
    \varphi\left(x, \frac{\beta}{|B|}\int_{\hat{B}} f \, dy \right)^{\frac{1}{p}}&\leq \varphi(x, 2\beta \max\{D_1, D_2\})^{\frac{1}{p}} \\ &\leq \varphi(x, 2\beta D_1)^{\frac{1}{p}}+\varphi(x, 2\beta D_2)^{\frac{1}{p}}
\end{align*}
for any $\beta>0$.
Adding the estimates for $f_1$ and $f_2$, we conclude the proof  by choosing $\beta:=\frac{1}{2}\min\{2^{-1}a^{-3/p}\beta_0^3\beta_1\beta'_J, \ \beta_J\beta_0^2\beta_2\}$.
\end{proof}

This key estimate gives us the following corollary.

\begin{cor}\label{corkey}
Let $\varphi \in \phiw(\Omega)$ satisfy \azero{}, \aonew{}, \atwow{} and \ainc{p} with $p\in (0, \infty)$. Let $\omega \in A_p$. Then there exists $\beta>0$ and $h\in L^1(\Omega,\omega)\cap L^{\infty}(\Omega)$ such that
\begin{equation*}
    \varphi(x,\beta Mf(x))^{\frac{1}{p}}\lesssim M(\varphi(\cdot,f)^{\frac{1}{p}})(x)+M(h^{\frac{1}{p}})(x)
\end{equation*}
 for every $f\in L^{\varphi}(\Omega,\omega)$ with $\varrho_\varphi^{\omega}(f)\leq 1$.
\end{cor}

Now we are ready to prove the main result.

\begin{thm}\label{lause}
    Let $\varphi \in \phiw(\Omega)$ satisfy \azero{}, \aonew{}, \atwow{} and \ainc{p}, with $p\in (1,\infty)$. If $\omega \in A_p$, then for all $f\in L^{\varphi}(\Omega, \omega)$,
    \begin{equation*}
        \|Mf\|_{L^\varphi(\Omega, \omega)}\lesssim \|f\|_{L^\varphi(\Omega, \omega)}.
    \end{equation*}
\end{thm}

\begin{proof}    
    Let us assume that $\norm{f}_{L^\varphi(\Omega, \omega)}< 1$. 
     From the unit ball property (Lemma \ref{ubp}) we have that $\varrho_{\varphi}^{\omega}(f)\leq 1$. 
    Then by Corollary \ref{corkey} we obtain
    \begin{equation*}
        \varphi(x,\beta  Mf(x))^{\frac{1}{p}}\lesssim M(\varphi(\cdot, f)^{\frac{1}{p}})(x)+M(h^{\frac{1}{p}})(x).
    \end{equation*}
    Multiplying both sides with $\omega^{1/p}$, raising  to the power $p$ and integrating, we find that 
    \begin{align*}
        \int_\Omega \varphi(x,\beta Mf(x)) \omega(x) \, dx \lesssim  \int_\Omega [M(\varphi(\cdot, f)^{\frac{1}{p}})&(x)]^p \omega(x) \, dx \\ 
        &+  \int_\Omega [M(h^{\frac{1}{p}})(x)]^p \omega(x) \, dx 
    \end{align*}
Since $\omega \in A_p$ we have by Lemma \ref{clas-max}  that
\begin{align*}
    \int_\Omega \varphi(x,\beta  Mf(x)) \omega(x) \, dx & \lesssim \int_\Omega \varphi(x, f)\omega(x)\, dx + \int_\Omega h(x)\omega(x)\, dx \\
    &= \varrho^\omega_{\varphi}( f) + \norm{h}_{L^1(\Omega, \omega)}.
\end{align*}
Thus $\varrho^\omega_\varphi(\beta Mf)\leq c(1+\norm{h}_{L^1(\Omega, \omega)})=: c_1$ and by (aInc)$_1$, $\varrho^\omega_\varphi(\frac{\beta}{a c_1}Mf)\leq1$. Again by the unit ball property we have that $\norm{Mf}_{L^\varphi(\Omega, \omega)}\leq \frac{ac_1}{\beta}\lesssim 1$ and the case $\norm{f}_{L^\varphi(\Omega, \omega)}< 1$ is complete. 

If $\norm{f}_{L^\varphi(\Omega, \omega)}\geq 1$, then we can reduce the claim to the previous case by considering the function $g:=\frac{f}{2\norm{f}_{L^\varphi(\Omega, \omega)}}$ with norm less than one.
\end{proof}

\section{Special cases}

\subsection*{Variable exponent Lebesgue spaces}

We obtain variable exponent \\ Lebesgue spaces when $\varphi(x,t)=t^{p(x)}$ for some measurable function $p:\Omega \to [1,\infty]$. Here we interpret $t^\infty:=\infty\chi_{(1,\infty)}(t)$ and $t^\frac{1}{\infty}:=\chi_{(0,\infty)}(t)$. Next, we present the standard assumptions for the exponent $p$: local $\log$-Hölder continuity and $\log$-Hölder decay. If we assume any weaker modulus of continuity than $\log$-Hölder continuity, then the maximal operator $M$ need not be bounded \cite{RPex}. In this sense the $\log$-Hölder continuity is necessary for boundedness of $M$. For more information on variable exponent spaces, see \cite{ves, cruz-uribe2013variable}.

\begin{defn}
If there exists $c_1>0$ such that the measurable function $p:\Omega \to [1, \infty]$ satisfies
$$\left|\frac{1}{p(x)}-\frac{1}{p(y)}\right|\leq \frac{c_1}{\log(\mathrm{e}+1/|x-y|)}$$ for every $x,y\in \Omega$, then we say that $\frac{1}{p}$ is \textit{locally} $\log$-Hölder \textit{continuous} on $\Omega$, $\frac{1}{p}\in C^{\log}$. 

If there exist $p_\infty\in[1,\infty]$ and $c_2>0$ such that $$\left|\frac{1}{p(x)}-\frac{1}{p_\infty}\right|\leq \frac{c_2}{\log(\mathrm{e}+|x|)}$$  for every $x \in \Omega$, then we say that $\frac{1}{p}$ satisfies the $\log$-Hölder \textit{decay condition}.
\end{defn}

 If $\frac{1}{p}$ satisfies both of the conditions above, we denote $\frac{1}{p}\in \mathcal{P}^{\log}(\Omega)$. We also denote $p^+_B:=\esssup_{x \in B\cap\Omega} p(x)$ and  $p^-_B:=\essinf_{x \in B\cap\Omega} p(x)$.

\begin{rem}
    The $\log$-Hölder decay condition can be weakened by a condition introduced by A. Nekvinda in \cite{Nekvinda}. The decay condition implies Nekvinda's condition which states that there exists $\beta>0$ such that 
    \begin{equation*}
        \int_\Omega\beta^{\frac{1}{|\frac{1}{p(x)}-\frac{1}{p_\infty}|}}\, dx < \infty.
    \end{equation*}
    This remains true also in the weighted case, as will be seen in Lemma \ref{wn}.
    
\end{rem}

Next we are going to link these conditions to the ones used in weighted generalized Orlicz space.

\begin{lem}[Lemma 7.1.1, \cite{Og}]
    Let $\frac{1}{p}\in \mathcal{P}^{\log}(\Omega)$. Then the $\Phi$-function $\varphi(x,t)=t^{p(x)}$ satisfies \azero{} and \inc{1}.
\end{lem}

\begin{lem}\label{wh}
Let $\frac{1}{p}\in \mathcal{P}^{\log}(\Omega)$ and $\omega \in A_\infty$. Then for all $B\subset \Rn$ we have
\[\omega(B)^{\frac{1}{p_B^+}-\frac{1}{p_B^-}}\lesssim 1.\]
\end{lem}

\begin{proof}
Let $B:=B(x,r)$. Since $\omega \in A_\infty$, there exists $q>1$ such that $\omega \in A_q$.
    By Lemma \ref{bxy}
\begin{equation*}\label{2.4}
        \omega(B(x,r))\gtrsim \omega(B(0,1))\left(\frac{r}{|x|+r+1}\right)^{qn}.
    \end{equation*}
    With estimate $\frac{r}{|x|+r+1}\geq \frac{r}{|x|+2\max\{r,1\}}\geq \frac{1}{2}\frac{r}{\max\{r,1\}}(|x|+1)^{-1}$ we have that
    \begin{align*}
        \omega(B(x,r))^{\frac{1}{p_B^+ }- \frac{1}{p_B^-}}\lesssim \omega(B(0,1))^{\frac{1}{p_B^+ }- \frac{1}{p_B^-}}(\tfrac{r}{\max\{r,1\}})^{\left(\frac{1}{p_B^+ }- \frac{1}{p_B^-}\right)qn}(|x|+1)^{\left(\frac{1}{p_B^-}-\frac{1}{p_B^+ }\right)qn}.
    \end{align*}
    Now the first term is a constant, second is either $1$ or bounded by $\log$-Hölder continuity, and the last term is bounded by $\log$-Hölder decay condition.
\end{proof}

\begin{lem}
    Let $\omega \in A_\infty$ and $\varphi(x,t)=t^{p(x)}$, where $\frac{1}{p}\in \mathcal{P}^{\log}(\Omega)$. Then $\varphi$ satisfies \aonew. 
\end{lem}

\begin{proof}
    Assume that $B$ is a ball with $\omega(B)\leq1$ and $x,y\in B\cap \Omega$. We need to show that there exists $\beta_1$ such that $\phi(x,\beta_1 t)\leq \phi(y,t)$ when $\phi(y,t) \in [1, \frac{1}{\omega(B)}]$. This is equivalent to showing that $t^{\frac{p(x)-p(y)}{p(x)}}\lesssim 1$ when $t^{p(y)}\in [1, \frac{1}{\omega(B)}]$. We may assume that $p(x)>p(y)$, since the other case follows directly from $t>1$. Then by Lemma \ref{wh}, 
    \begin{equation*}
        t^{\frac{p(x)-p(y)}{p(x)}}\leq\omega(B)^{\frac{p(y)-p(x)}{p(x)p(y)}}\leq \omega(B)^{p_B^--p_B^+}\lesssim 1,
    \end{equation*}
    and the claim follows.
\end{proof}
The next lemma shows that the weighted Nekvinda decay condition follows.

\begin{lem}\label{wn}
    Let $\omega \in A_\infty$ and $\varphi(x,t)=t^{p(x)}$, where $\frac{1}{p}$ satisfies the $\log$-Hölder decay condition. Then there exists $\lambda\in(0,1)$ such that 
    \begin{equation*}
        \int_\Omega\lambda^{\frac{1}{|\frac{1}{p(x)}-\frac{1}{p_\infty}|}}\omega(x)\, dx < \infty.
    \end{equation*} 
\end{lem}
\begin{proof}
    By the $\log$-Hölder decay condition, $|\tfrac{1}{p(x)}-\tfrac{1}{p_\infty}|\leq \frac{c}{\log(e+|x|)}$.  Since $\omega \in A_\infty$, there exists $q>1$ such that $\omega \in A_q.$ Then by Lemma \ref{bxy}, $\omega(B(0,r))\lesssim r^{qn}$ when $r>1$. It follows that
\begin{align*}
 \int_\Omega\lambda^{\frac{1}{|\frac{1}{p(x)}-\frac{1}{p_\infty}|}}\omega(x)\, dx&\leq \sum_{j=1}^\infty\int_{B(0,j)\setminus B(0,j-1)}\lambda^{\frac{1}{|\frac{1}{p(x)}-\frac{1}{p_\infty}|}}\omega(x)\, dx \\
&\leq \sum_{j=1}^\infty\lambda^{c\log(e+j-1)}\int_{B(0,j)\setminus B(0,j-1)}\omega(x)\, dx \\
&\leq \sum_{j=1}^\infty j^{c\log\lambda}\omega(B(0,j)) \\
&\lesssim\sum_{j=1}^\infty j^{c\log\lambda+qn}<\infty,
\end{align*}
by choosing $\lambda<e^{-(\frac{qn+1}{c})}$.
\end{proof}

\begin{lem}
    Let $\omega \in A_\infty$ and $\varphi(x,t)=t^{p(x)}$, where $\frac{1}{p}$ satisfies the $\log$-H\"older decay condition.. Then $\varphi$ satisfies \atwow. 
\end{lem}
\begin{proof}
    By Lemma \ref{wn}, there exist $\lambda\in(0,1)$ and $p_\infty \in [1,\infty]$ such that 
    \begin{equation*}
        \int_\Omega\lambda^{\frac{1}{|\frac{1}{p(x)}-\frac{1}{p_\infty}|}}\omega(x)\, dx < \infty.
    \end{equation*}
    Let us assume that $\phi_\infty(t):=t^{p_\infty} \leq 1$. If $p(x)\geq p_{\infty}$, then $(\lambda t)^{p(x)}\leq t^{p_\infty}$. When $p(x)<p_\infty<\infty$, it follows from Young's inequality that 
    \begin{equation*}
        (\lambda t)^{p(x)} \leq t^{p_\infty} + \lambda^{\frac{1}{|\frac{1}{p(x)}-\frac{1}{p_\infty}|}}.
    \end{equation*}
    If $p_\infty=\infty$, then $t^{p_\infty}=0$ and $\lambda^{\frac{1}{|\frac{1}{p(x)}-\frac{1}{p_\infty}|}}=\lambda^{p(x)}$. The other case, where $\phi(x,t)=t^{p(x)}\leq 1$, is analogous. By Lemma \ref{a2eq} with \begin{equation*}
        h(x):=\lambda^{\frac{1}{|\frac{1}{p(x)}-\frac{1}{p_\infty}|}},
    \end{equation*}
    $\varphi$ satisfies \atwow{}.\end{proof}

\subsection*{Double phase spaces}

In the double phase case $\phi(x,t)=t^p+a(x)t^q$, with $1\leq p<q<\infty$, the central issue is the behavior 
of $a$ around the zero set $\{x: a(x)=0\}$. M. Colombo and C. Mingione \cite{ColM15a} 
found that the critical H\"older exponent with which $a$ must
approach zero is $\frac np (q-p)$, which also gives a sufficient
condition for (A1) to hold \cite{H15}. A similar observation applies in the weighted double phase case.

\begin{lem}
     Let $\varphi(x,t)=t^p+a(x)t^q$ with $1 \leq p < q < \infty $ and $a\in L^\infty(\Omega)$ non-negative. Then $\varphi$ satisfies \azero{}, \atwow{}, and \inc{p}.
\end{lem}

\begin{proof}
    The proof is the same as that of Proposition 7.2.1 from \cite{Og}.
\end{proof}
   
\begin{lem}
    Let $\varphi(x,t)=t^p+a(x)t^q$ with $1 \leq p < q < \infty $ and $a\in L^\infty(\Omega)$ be non-negative. Then $\varphi$ satisfies \aonew{} if and only if 
    \begin{equation*}
        a(x) \lesssim a(y) + \omega(B)^{\frac{q-p}{p}}
    \end{equation*}
    for every $x,y \in B\cap \Omega.$
\end{lem}

\begin{proof}
    We may notice that $\varphi(x,t)\approx \max\{t^p, a(x)t^q\}$. Then \aonew{} becomes
    \begin{equation*}
        \max\{\beta_1^pt^p, a(x)\beta_1^qt^q\}\leq\max\{t^p, a(y)t^q\}
    \end{equation*}
    for $x,y \in B\cap \Omega$, $\omega(B)\leq1$ and $\varphi(y,t)\in [1, \frac{1}{\omega(B)}]$.
Dividing the inequality by $t^q$ we get
 \begin{equation*}
        \max\{\beta_1^pt^{p-q}, a(x)\beta_1^q\}\leq\max\{t^{p-q}, a(y)\}.
    \end{equation*}
    This holds trivially if $a(x)\beta_1^q \leq \beta_1^p t^{p-q}$, and thus the condition is equivalent to 
     \begin{equation*} a(x)\beta_1^q\leq\max\{t^{p-q}, a(y)\}
    \end{equation*}
for $a(x)\beta_1^q > \beta_1^p t^{p-q}$.

    From $\varphi(y,t)\in [1, \frac{1}{\omega(B)}]$ it follows that $t\leq \frac{1}{\omega(B)^{1/p}}$. Since $p-q<0$ it suffices to use only the upper bound for $t$. Thus, the inequality above becomes
    \begin{equation*}
        a(x)\lesssim \max\{a(y), \omega(B)^{\frac{q-p}{p}}\} \approx a(y) + \omega(B)^{\frac{q-p}{p}}. \qedhere
    \end{equation*} 
\end{proof}

\bibliography{ref}
\bibliographystyle{plain}

\end{document}